\documentclass[12pt]{article}
\usepackage{latexcad}
\input{amssym}

\newtheorem{prethm}{{\bf Theorem}}

\newenvironment{thm}{\begin{prethm}{\hspace{-0.5
em}{\bf.}}}{\end{prethm}}
\newtheorem{precor}{{\bf Corollary}}

\newenvironment{cor}{\begin{precor}{\hspace{-0.5
em}{\bf.}}}{\end{precor}}
\newtheorem{preprop}{{\bf Proposition}}

\newenvironment{prop}{\begin{preprop}{\hspace{-0.5
em}{\bf.}}}{\end{preprop}}
\newtheorem{preque}{{\bf Question}}

\newenvironment{que}{\begin{preque}{\hspace{-0.5
em}{\bf.}}}{\end{preque}}
\newtheorem{preques}{{\bf Question}}

\newtheorem{prelemma}{{\bf Lemma}}

\newtheorem{prefact}{{\bf Fact}}

\newtheorem{preobs}{{\bf Observation}}

\newenvironment{obs}{\begin{preobs}{\hspace{-0.5
em}}}{\end{preobs}}
\newtheorem{prefig}{{\bf Figure}}

\newtheorem{prelemm}{{\bf Lemma}}

\newtheorem{preex}{{\bf Example}}

\newtheorem{prepro}{{\bf Proposition}}

\newtheorem{prelem}{{\bf Theorem}}

\newenvironment{lem}{\begin{prelem}{\hspace{-0.5
em}{\bf.}}}{\end{prelem}}
\newtheorem{preproof}{{\bf Proof.}}

\newenvironment{proof}[1]{\begin{preproof}{\rm
               #1}\hfill{$\rule{2mm}{2mm}$}}{\end{preproof}}

\newtheorem{preconj}{{\bf Conjecture}}

\newenvironment{conj}{\begin{preconj}{\hspace{-0.5
em}{\bf.}}}{\end{preconj}}
\newtheorem{predeff}{{\bf Definition}}

\newenvironment{deff}{\begin{predeff}{\hspace{-0.5
em}{\bf.}}}{\end{predeff}}
\def\n#1{\vbox to 3mm{\vspace{1mm}\vfill \hbox to 2.0mm{\hfill
             $#1$\hfill} \vfill }}
\def\m#1#2{\raise 0.2ex\hbox{
    ${#1_{\displaystyle #2}}$}}
\def\x#1{\raise 0.5ex\hbox{
    ${#1}$}}
\def\arraystretch{1.0}                 

\def\m#1#2#3{\raise .8ex\hbox{
     ${#1_{{\bf \displaystyle #2}_{\displaystyle #3}}}$}}

\usepackage{array}

\def\newpic#1{}
\date{}

\begin{document}

\title{
{\Large{\bf On the simultaneous edge coloring of graphs}}}
%

{\small
\author{
{\sc Behrooz Bagheri Gh.} and
{\sc Behnaz Omoomi}\\
[1mm]
{\small \it  Department of Mathematical Sciences}\\
{\small \it  Isfahan University of Technology} \\
{\small \it 84156-83111, Isfahan, Iran}}


 \maketitle \baselineskip15truept


\begin{abstract}
 A {\sf $\mu$--simultaneous edge coloring} of graph $G$ is a set of $\mu$ proper edge colorings  of
$G$ with a same color set  such that for each vertex, the sets of colors appearing on the edges incident to that vertex are the same
in each coloring and
no edge receives the same color in any two colorings.
The $\mu$--simultaneous edge coloring of bipartite graphs has a close relation with $\mu$--way Latin trades.
  Mahdian $\rm et\ al.$ (2000)  conjectured that every bridgeless bipartite  graph is $2$--simultaneous edge colorable.
Luo $\rm et\ al.$ (2004) showed that  every bipartite graphic sequence $S$ with all its elements greater than one, has a
realization that admits a $2$--simultaneous edge coloring.
In this paper, the $\mu$--simultaneous edge coloring of graphs is studied. Moreover, the properties of the extermal counterexample
to the above conjecture are investigated. Also, a relation between
$2$--simultaneous edge coloring of a graph with a cycle double cover with certain properties is shown and using this relation, some results about  $2$--simultaneous edge colorable graphs are obtained.
\end{abstract}

\section{\large Introduction}

 In this paper all graphs we consider are finite and simple. For notations and definitions  we refer to~\cite{Bondy}.
This section deals with a brief review of  some concepts related to the main subject of the paper.

 Let $S$ be a nonempty proper subset of $V(G)$. The subset $[S,\overline{S}]=\{uv\in E(G) : u\in S, v\in \overline{S}\}$ of $E(G)$ is called an {\sf edge cut}.  A {\sf $k$-edge cut} is an
  edge cut $[S,\overline{S}]$, where $|[S,\overline{S}]|=k$.  An edge cut  $F$, is called {\sf trivial} if one of the component in $G\setminus F$ be an isolated vertex.
The {\sf edge connectivity} of $G$, $\kappa'(G)$, is the minimum $k$ for which  $G$ has a $k$-edge cut and $G$ is said to be {\sf $k$--edge-connected} if $\kappa'(G)\ge k$. A $2$--edge-connected graph is called a {\sf bridgeless} graph.

A {\sf proper edge coloring} of a  graph $G$ is a labeling from $E(G)$ to the color set $[l]=\{1,\ldots,l\}$ such that incident edges have different colors.
The {\sf edge chromatic number} of $G$, $\chi'(G)$, is the least $l$ such that $G$ admits a proper edge coloring with label set $[l]$.
 A {\sf $k$-factor} of graph $G$ is a $k$-regular spanning subgraph of $G$, and $G$ is  {\sf $k$-factorable} if there are edge disjoint
$k$-factors $H_1,\ldots,H_l$ such that $G=H_1\cup\ldots \cup H_l$. Note that an $r$-regular graph $G$ is $1$-factorable if and only if
$\chi'(G)=r$.

We use the term {\sf circuit} for a connected $2$-regular graph and the term {\sf cycle} for a graph that all its vertices  have even degrees.
A {\sf cycle double cover ($\rm CDC$)}, $\mathcal{C}$, of a  graph $G$  is a collection of its cycles  such that every
 edge of $G$ is contained in precisely  two cycles in $\mathcal{C}$ and a  {\sf $k$-cycle double cover ($k$-$\rm CDC$)} of $G$ is a $\rm CDC$ of  $G$ such that consisting of at most $k$ cycles of $G$.
Note that the
cycles are not necessarily distinct.
A necessary condition for a graph to have a $\rm CDC$ is
the bridgeless property. Seymour~\cite{Sey} in $1979$ conjectured
that this condition is also sufficient.
\begin{conj}\label{conj:CDC}~{\rm\cite{Sey}} {\rm ($\rm CDC$ conjecture)}
Every bridgeless graph has a $\rm CDC$.
\end{conj}

No counterexample to the $\rm CDC$
conjecture is known. It is proved that the minimal counterexample to the $\rm CDC$
conjecture is a bridgeless cubic graph with edge chromatic number equal to $4$, which is called a {\sf snark}.
The {\rm CDC} conjecture has many stronger forms, one of which is the
following conjecture. An {\sf oriented cycle double cover ($\rm OCDC$)} of a  graph $G$  is a $\rm CDC$ of $G$ in which
every circuit can be oriented in such a way that every edge of the graph is covered by two directed
circuits in two different directions.
\begin{conj}\label{conj:OCDC}~{\rm\cite{Survey}} {\rm ($\rm OCDC$ conjecture)}
 Every bridgeless  graph  has an $\rm OCDC$.
\end{conj}

The concept of cycle double cover has a relation with nowhere-zero flow in graphs. Some necessary relations of these two concepts are presented in what follows.\\
Let $G$ be a simple graph and $(D,f)$ be an ordered pair, where $D$ is
an orientation of $E(G)$ and $f$ is a weight on $E(G)$ to $\Bbb{Z}$.
For each $v\in V(G)$, denote
$$f^+(v)=\sum{f(e)}\ \ \ \ \ \mbox{and}\ \ \ \ \ f^-(v)=\sum{f(e)},$$
where the summation is taken over all directed edges of $G$ (under the orientation $D$)
with tails and heads, respectively, at the vertex $v$.
An {\sf integer flow} of $G$ is an ordered pair $(D,f)$ such that  for every vertex $v\in V(G)$, $f^+(v)=f^-(v)$.
The support of $f$, $supp(f)$, is the set of the edge $e\in E(G)$ that $f(e)\ne 0$.
 A {\sf nowhere-zero $k$-flow} of $G$ is an integer flow $(D,f)$ such that $supp(f)=E(G)$ and $-k< f(e)< k$, for every $e\in E(G)$ and  is denoted by $k$-$\rm NZF$.
\begin{lem}{\rm~\cite{Zhang'sBook}}\label{thm:4-NZF}
\item{\rm (i)}
 If every edge of a graph $G$ is contained in a circuit of length
at most $4$, then $G$ admits a $4$-$\rm NZF$.
\item{\rm (ii)}
 A graph $G$ admits a $4$-$\rm NZF$ if and only if $G$ has a $4$-$\rm CDC$.
\item{\rm (iii)}
A graph $G$ admits a $4$-$\rm NZF$ if and only if
$G$ has an $\rm OCDC$ consists of four directed cycles.
\end{lem}
%
%

%

Let $G$ be a bipartite graph with bipartition $(X,Y)$.
The bipartite degree sequence
of $G$ is the sequence $(x_1,x_2,\ldots,x_n;y_1, y_2,\ldots,y_m)$, where $(x_1,x_2,\ldots,x_n)$ are the vertex
degrees in $X$ and $(y_1, y_2,\ldots,y_m)$ are the vertex degrees in $Y$. A sequence $S$ of positive
integers is called a {\sf bipartite graphic sequence} if there exists a bipartite graph $G$ whose
bipartite degree sequence is $S$; if so then the graph $G$ is called a {\sf realization} of $S$.

%
\begin{deff}\label{def:k--SE}~{\rm\cite{Mahdian}}
A {\sf $\mu$--simultaneous edge coloring} of  graph $G$ is a set of $\mu$ proper edge colorings of $G$ with the color set $[l]$, say $(c_1,c_2,\ldots,c_{\mu})$, such that
\item{$\bullet$} for each vertex, the sets of colors appearing on the edges incident to that vertex are the same
in each coloring;
\item{$\bullet$} no edge receives the same color in any two colorings.\\ \\
If $G$ has a $\mu$--simultaneous edge coloring, then $G$ is called a {\sf $\mu$--simultaneous edge colorable} graph.
The minimum $l$ that there exists a $\mu$--simultaneous edge coloring of $G$ with the color set $[l]$, is called {\sf $\mu-SE$ chromatic number} of $G$ and denoted by $\chi'_{\mu-SE}(G)$.
\end{deff}
Note that in every $\mu$--simultaneous edge coloring of a graph $G$, $\mu \le \deg_G(v)$, for every $v\in V(G)$, because every edge $e=uv\in E(G)$ admits $\mu$  different colors of colors appeared of the edge incident to  $v$.
\begin{obs}
 If $G$ is a $\mu$--simultaneous edge colorable graph, then $\mu\le \delta(G)$, where $\delta(G)$ is the minimum degree of $G$. Moreover,
$$\Delta(G)\le \chi'(G)=\chi'_{1-SE}(G)\le \chi'_{2-SE}(G)\le \cdots \le \chi'_{\mu-SE}(G).$$
\end{obs}
There are some graphs $G$  that $\chi'(G) < \chi'_{\mu-SE}(G)$; for example in the next section we show that for graph $G$ shown
 in Figure~\ref{HGR}, $\chi'_{2-SE}(G)\le 4$, and by a case study, it can be checked that $G$ has no $2$--simultaneous edge coloring with $3$ colors. Thus, $\chi'_{2-SE}(G)=4$ while $\chi'(G)=3$. In this case, $\chi'_{2-SE}(G)=\Delta(G)+1$. This is a natural question: Is this true that $\Delta(G)\le \chi'_{2-SE}(G)\le \Delta(G)+1$?
 \begin{figure}[ht]
\centering \unitlength=1mm
\begin{picture}(44,36)
\thicklines
\drawthickdot{6.79}{30.16}
\drawthickdot{6.79}{22.16}
\drawthickdot{6.79}{14.16}
\drawthickdot{6.79}{6.16}
\drawthickdot{38.83}{30.16}
\drawthickdot{38.83}{22.16}
\drawthickdot{38.83}{14.16}
\drawthickdot{38.83}{6.16}
\drawpath{38.83}{30.16}{6.79}{30.16}
\drawpath{6.79}{30.16}{38.83}{22.16}
\drawpath{38.83}{22.16}{6.79}{14.16}
\drawpath{6.79}{14.16}{38.83}{6.16}
\drawpath{38.83}{6.16}{6.79}{6.16}
\drawpath{6.79}{6.16}{38.83}{30.16}
\drawpath{6.79}{14.16}{38.83}{14.16}
\drawpath{38.83}{14.16}{6.79}{22.16}
\drawpath{6.79}{22.16}{38.83}{30.16}
\drawpath{38.83}{22.16}{6.79}{22.16}
\end{picture}
 \vspace*{-6mm}\caption{\label{HGR} \scriptsize{$\chi'(G)=3$ and $\chi'_{2-SE}(G)=4$.}}
\end{figure}
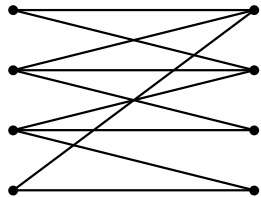
%

 At the $16^{\rm {th}}$ British Combinatorial Conference ($1997$), Cameron introduced the concept
{\sf $2$--simultaneous edge coloring}. He use this concept to reformulate a conjecture of  Keedwell ($1994$)
on the existence of critical partial Latin squares of given type. In fact he conjectured (called {\sf SE conjecture}) that
for each bipartite graphic sequence $S$ with all its elements greater than one, there exists
 a $2$--simultaneous edge colorable realization.

 Mahdian $\rm et\ al.$ in~\cite{Mahdian} showed that the  $2$--simultaneous edge coloring of
every bipartite graph is  equivalent to an $\rm OCDC$ of that graph.
Also, they  conjectured that every bridgeless bipartite  graph is $2$--simultaneous edge colorable.
\begin{lem}~{\rm ~\cite{Mahdian}}\label{SEOCDC}
Every bipartite graph $G$ is $2$-simultaneous edge colorable if and only if $G$ has an $\rm OCDC$.
\end{lem}
\begin{conj}\label{conj:Strong2--SE}{\rm~\cite{Mahdian}} {\rm(strong SE conjecture)}
 Every bridgeless bipartite
graph is   $2$--simultaneous edge colorable.
\end{conj}
Luo $\rm et\ al.$ in~\cite{4NZF--SE} showed that  every bipartite graphic sequence $S$ with all its elements greater than one, has a
realization that admits a $4$-$\rm NZF$. Thus, by Theorems~\ref{thm:4-NZF}~(iii) and~\ref{SEOCDC}, they proved that the SE conjecture is true.
%

In Section $2$, we see the relation between $\mu$--simultaneous edge coloring and $\mu$--way Latin trade,
also, we give some sufficient conditions for graphs to be  $\mu$--simultaneous edge colorable. In Section $3$, we consider the case
$\mu=2$. First, some properties for the extermal counterexample to the strong SE conjecture are given; then, we discuss on $2$--simultaneous edge coloring  for general graphs and introduce some  $2$--simultaneous edge colorable graphs and some graphs which has no $2$--simultaneous edge coloring.

\section{\large $\mu$--simultaneous edge coloring and $\mu$--way Latin trade}
\noindent A {\sf partial Latin square} $P$ of order $n$ is an
$n\times n$ array of elements from the set $[n]=\{1,\ldots,n\}$, where each
element of $[n]$ appears at most once in each row and at most once
in each column.  We can represent each  partial Latin square, $P$, as a subset of $[n]\times [n]\times [n]$,
$$P=\{(i,j;k) : \mbox{ element $k$ is located in position }(i,j)\}.$$
The partial latin square $P$ is called {\sf symmetric} if $(i,j;k)\in P$ if and only if $(j,i;k)\in P$.
The set $S_P=\{(i,j) : (i,j;k)\in P,\ 1\le k\le n \}$ of the
partial Latin square $P$ is called the {\sf shape } of $P$ and
$|S_P|$ is called the {\sf volume} of $P$. By ${\cal R}_P^i$ and
${\cal C}_P^j$ we mean the set of entries in row $i$ and column
$j$, respectively of $P$.

A {\sf $\mu$--way Latin trade},
$(T_1,\ldots,T_{\mu})$, of volume $s$ is a collection of $\mu$ partial Latin squares
$T_1,\ldots,T_{\mu}$, containing exactly the same $s$ filled
cells, such that if cell $(i, j)$ is filled, it contains a
different entry in each of the $\mu$ partial Latin squares, and
 row $i$ in each of the $\mu$ partial Latin squares
contains, set-wise, the same symbols and column $j$, likewise. If
$\mu=2$, $(T_1,T_2)$ is called a {\sf Latin bitrade}.
The {\sf volume spectrum } ${\cal S}_\mu$ for all $\mu$--way Latin trades is
the set of possible volumes of $\mu$--way Latin trades.
 For a survey on this
topic see \cite{MR2041871},
 \cite{MR2453264},  and \cite{MR2048415}.

For every  $\mu$--way Latin trade $T=(T_1,\ldots,T_{\mu})$ of volume $s$ there exists a $\mu$--simultaneous edge colorable  bipartite graph $G$ with $s$ edges and bipartite degree sequence
$S=(|{\cal R}_T^1|,\ldots,$ $|{\cal R}_T^n|;|{\cal C}_T^1|,\ldots,|{\cal C}_T^m|)$. In fact $G=(X,Y)$ is a bipartite graph, where $X=\{x_1,\ldots,x_n\}$
and $Y=\{y_1,\ldots,y_m\}$ such that for every filled cell $(i,j)$ in $T$, there is an edge between $x_i$ and $y_j$ and the element that located in position $(i,j)$ of $T_k$ is the color of edge $x_iy_j$ in the $k^{\mbox{th}}$ coloring of $\mu$--simultaneous edge coloring of $G$, for $1\le k\le \mu$.

In general, for every symmetric $\mu$--way Latin trade $T=(T_1,\ldots,T_{\mu})$ of volume $s$ that $(i,i)\notin S_T$, for every $i$, there exists a  $\mu$--simultaneous edge colorable  graph $G$ with $\frac{1}{2}s$ edges and degree sequence
$S=(|{\cal R}_T^1|,\ldots,$ $|{\cal R}_T^n|)$. In fact $G$ is a  graph, where $V(G)=\{v_1,\ldots,v_n\}$
 such that for every filled cell $(i,j)$ in $T$, there is an edge between $v_i$ and $v_j$ and the element that located in position $(i,j)$ of $T_k$ is the color of edge $v_iv_j$ in the $k^{\mbox{th}}$ coloring of $\mu$--simultaneous edge coloring of $G$, for $1\le k\le \mu$.

In Figure~\ref{Latintrade} a  Latin bitrade, $T=(T_1,T_2)$, of volume $10$ is demonstrated. ($\bullet$ means the cell is empty.)
In fact, $T$ is the  Latin bitrade corresponding to a $2$--simultaneous edge coloring of the graph  $G$ that showed in Figure~\ref{HGR}. Therefore, $\chi'_{2-SE}(G)\le 4$.
\def\arraystretch{1.3}
\begin{center}
\begin{tabular}{|c|c|c|c|}\hline
1&2&$\bullet$&$\bullet$\\\hline
2&4&3&$\bullet$\\\hline
$\bullet$&1&4&3\\\hline
3&$\bullet$&$\bullet$&1\\\hline
\end{tabular}
\hspace*{10mm}
\begin{tabular}{|c|c|c|c|}\hline
2&1&$\bullet$&$\bullet$\\\hline
3&2&4&$\bullet$\\\hline
$\bullet$&4&3&1\\\hline
1&$\bullet$&$\bullet$&3\\\hline
\end{tabular}
\end{center}

\begin{center}
\begin{figure}[ht]
\vspace*{-7mm} \caption{\label{Latintrade}\scriptsize{$T=(T_1,T_2)$ a Latin bitrade of volume $10$.}}
\end{figure}
\end{center}
Since Luo $\rm et\ al.$ in~\cite{4NZF--SE} showed that  each bipartite graphic sequence $S$ with all its elements greater than $1$, has a $2$--simultaneous edge colorable bipartite realization, we have $${\cal S}_2=\Bbb{N}\setminus \big \{1,2,3,5\}.$$
\begin{lem}{\rm~\cite{MR1904722,MR1874724}}\label{spec3}
The volume spectrums  for all $\mu$--way Latin trades, $\mu=3,4,5$ are
$${\cal S}_3=\Bbb{N}\setminus \big([1,8]\cup \{10,11,13,14\}\big);$$
$${\cal S}_4=\Bbb{N}\setminus ([1,15]\cup \{17,18,19,21,22,26\});$$
$${\cal S}_5=\Bbb{N}\setminus ([1,24]\cup [26,29] \cup \{31,32,33,37,38\}).$$
\end{lem}

Let $S=(3,3,3,4;3,3,3,4)$ be a bipartite graphic sequence.
By Theorem~\ref{spec3}, there is no $3$--way Latin trade of volume $3+3+3+4=13$. Thus, the bipartite graph $G=(X,Y)$ with $X=\{x_1,x_2,x_3,x_4\}$ and
$Y=\{y_1,y_2,y_3,y_4\}$ and $E(G)=\{x_iy_j : 1\le i\ne j\le 4\}\cup \{x_4y_4\}$ is not
$3$--simultaneous edge colorable. Note that $G$ is a $3$--edge-connected bipartite graph. Therefore, the generalization of the  strong SE conjecture and SE conjecture are not true.

One can be asked the following two natural questions related to this concept.
\begin{que}\label{Sk-SE}
Is there a positive integer $s_{\mu}$ such that every $\mu$--edge-connected bipartite
graph with at least $s_\mu$ edges admits a  $\mu$--simultaneous edge coloring?
\end{que}

\begin{que}\label{k-SE}
Is there a positive integer $s_{\mu}$ such that
each bipartite graphic sequence $S=(x_1,\ldots,x_n;y_1,\ldots,y_m)$ with all its elements greater than $\mu-1$ and $\sum_{1\le i\le n}x_i\ge s_{\mu}$, there exists a   $\mu$--simultaneous edge colorable bipartite realization?
\end{que}

In~\cite{Edmonds}, Edmonds showed that every graphic degree sequence, with all degrees at least $\mu\ge 2$, has a $\mu$--edge-connected realization. In~\cite{Hajiaghaee}, Hajiaghaee $\rm et\ al.$  proved that every bipartite graphic sequence, with all degrees at least $2\mu\ (\mu\ge 1)$, has a $2\mu$--edge-connected realization.
In the following theorem we prove
 a generalization of these theorems;   every bipartite graphic sequence, with all elements greater than $\mu -1$, has a $\mu$--edge-connected bipartite realization.
Therefore, if the response of Question~\ref{Sk-SE} is positive, then the response of Question~\ref{k-SE} is also positive.
For this purpose we need the following theorem.
\begin{lem}{\rm~\cite{Mahdian}}\label{2realization}
For every bipartite graphic sequence $S$ with all its elements greater than one,  there
exists a $2$--edge-connected realization.
\end{lem}
\begin{thm}\label{thm:realization}
Every bipartite graphic sequence $S$ with all its elements greater than $\mu-~1,$ $\mu\ge ~3$,
has a $\mu$--edge-connected realization.
\end{thm}
\begin{proof}{
Let  $r$ be the maximum edge connectivity
among all realizations of the bipartite graphic sequence $S$ and $r\le \mu-1$. By Theorem~\ref{2realization}, $r\ge 2$.
 Also, let $G=(X,Y)$ be a bipartite realization of $S$ with the edge connectivity $\kappa^{'}(G)=r,$ and $G$
 has the minimum number of $r$-edge cuts.   Assume that $F=\{e_1,e_2,\ldots,e_r\}$ is an $r$-edge cut of $G$. Therefore, $G\setminus F$ has exactly two components $G_1$ and $G_2$.\\
First, we show that $G_1$ and $G_2$ are bridgeless. Otherwise, without loss of generality, assume that $e=uv\in E(G_1)$ is a cut edge of $G_1$ and $G_{11}$ and $G_{12}$ are components of $G_1\setminus \{e\}$.\\
 If $r=2$ and $S_1$ is the bipartite degree sequence of $G_1$, then 
  by Theorem~\ref{2realization}, there is a bridgeless bipartite graph $G_1'$ with the degree sequence $S_1$. Thus, $G'=(G\setminus E(G_1))\cup E(G_1')$ is a realization of $S$ with the same edge connectivity as $G$ and the number of its $r$-edge cuts is less than the  number of $r$-edge cuts of $G$, which is a contradiction.  \\
If $r\ge 3$ and $F_i$ is the edges between $G_{1i}$ and $G_2$, $i=1,2$, then $F=F_1\cup F_2$ and for some $i$, say $i=2$, $|F_2|\ge 2$.
  Therefore, $F'=F_1\cup\{e\}$ is an edge cut of size at most $r-1$, which is a contradiction.
  Thus, $G_1$ and $G_2$ are bridgeless.
  Hence, in the bridgeless components $G_1$ and $G_2$, every edge lies in a circuit.\\
 Since $\delta(G)\ge \mu$,  for every
$v_i\in V(G_i),\ i=1,2,$ there exists a vertex $v'_i\in V(G_i)\cap N_G(v)$ such that $N_G(v'_i)\subseteq V(G_i)$; so,
there exists a vertex $v''_i\in V(G_i)\cap N_G(v'_i)$ such that $N_G(v''_i)\subseteq V(G_i)$.\\
Let $v_i\in V(G_i)\cap X$ and $C_i$ be a circuit in $G_i$ such that $e_i=v'_iv''_i\in E(C_i)$, $i=1,2$.
 Now by switching two edges $e_1$ and $e_2$ with two edges
$v'_1v''_2$ and $v'_2v''_1$, we obtain a bipartite graph $G'$ with the same degree sequence as $G$ in which
 $F$ is not an $r$-edge cut anymore, and  no new  $r$-edge cut is appeared.  This contradicts
the minimality of the number of $r$-edge cuts in $G$.
Therefore, $r\ge \mu$ and this complete the proof.
}\end{proof}
Mahdian $\rm et\ al.$ showed that there exists an infinite
family of $\mu$--simultaneous edge colorable graphs.
In the rest of this section we consider  $\mu$--simultaneous edge colorings of complete graphs, complete bipartite graphs and
some graph operations such as join and graph product.
\begin{lem}~{\rm\cite{Mahdian}}\label{r-regular}
Every $r$-regular $1$-factorable graph is $\mu$--simultaneous edge colorable for
every $\mu\le r$.
\end{lem}

For example every complete graph $K_{2l}$, $l\ge 2$, is $\mu$--simultaneous edge colorable for
every $\mu\le 2l-1$; every complete bipartite graph $K_{n,n}$, $n\ge 2$, is $\mu$--simultaneous edge colorable for
every $\mu\le n$; every complete multipartite graph $K_{r_1,r_2,\ldots,r_n}$, when $r_1=\cdots=r_n=r,n\ge 2$, and $rn$ is even,
 is $\mu$--simultaneous edge colorable for every $\mu\le (n-1)r$; and every hypercube graph $Q_n$, $n\ge 1$, is $\mu$--simultaneous edge colorable for every $\mu\le n$.
\begin{lem}{\rm~\cite{MR1904722}}
If $\min\{m, n\}\ge  \mu$, then there exists an $m \times n$ $\mu$-way Latin trade of
volume $mn$.
\end{lem}
\begin{cor}\label{K_r,s}
Every $K_{n,m}$ admits a $\mu$--simultaneous edge coloring, for $\mu \le  \min\{m, n\}$ and $n,m\ge 2$. Moreover,  $\chi'_{\mu-SE}(K_{n,m})=\max \{m, n\}$.
\end{cor}

The {\sf join} of two simple graphs $G$ and $H$, $G\vee H$, is the graph obtained from
the disjoint union of $G$ and $H$ by adding the edges $\{uv :  u\in V(G),\ v\in V(H)\}$.
\begin{thm}\label{join}
Let $G_i$ be a $\mu$--simultaneous edge colorable graph of order $n_i\ge 2$. The join graph $G_1 \vee G_2$ has a $\mu$--simultaneous edge coloring.
\end{thm}
\begin{proof}{
Since $G_i$'s has a $\mu$--simultaneous edge coloring, for $\mu \le \min \{n_1,n_2\}$, by Corollary~\ref{K_r,s},
 $K_{n_1,n_2}$ has a $\mu$--simultaneous edge coloring. Now we
define  a $\mu$--simultaneous edge coloring of $G_1 \vee G_2$ by a $\mu$--simultaneous edge coloring of  the copy $G_i$ in
$G_1 \vee G_2$ with the color set $\{1,\ldots,\chi'_{\mu-SE}(G_i)\}$, $i=1,2$, and a
$\mu$--simultaneous edge coloring of the copy $K_{n_1,n_2}$ in $G_1 \vee G_2$ with the color set $\{r+1,\ldots,r+\Delta(K_{n_1,n_2})\}$, where
$r=\max \{\chi'_{\mu-SE}(G_1),\chi'_{\mu-SE}(G_2)\}$.
}\end{proof}
\begin{prop}\label{prop:K_7}
The complete graphs $K_7$ and $K_9$ admit a $\mu$--simultaneous edge coloring, for $\mu=2,3$.
\end{prop}
\begin{proof}{
Let $V(K_7)=\{v_1,v_2,\ldots,v_7\}$ be the vertex set of $K_7$.
The following colorings, $(c_1,c_2,c_3)$, is a $3$--simultaneous edge coloring of $K_7$,  where $c_\mu$, is a proper edge coloring of $K_7$ with color set $\{1,2,\ldots,7\}$, and  $v_iv_j: l_1,l_2,l_3$ means $c_{\mu}(v_iv_j)=l_\mu$,  $\mu=1,2,3$.\\
$v_1v_2:5,7,6;\ \ v_1v_3:2,3,1;\ \ v_1v_4:3,2,7;\ \ v_1v_5:6,1,5;\ \ v_1v_6:7,6,3;\ \ v_1v_7:1,5,2;$\\
$v_2v_3:7,2,5;\ \ v_2v_4:6,1,4;\ \ v_2v_5:1,6,7;\ \ v_2v_6:4,5,2;\ \ v_2v_7:2,4,1;$\\
$v_3v_4:1,7,2;\ \ v_3v_5:4,5,3;\ \ v_3v_6:5,4,7;\ \ v_3v_7:3,1,4;$\\
$v_4v_5:7,4,1;\ \ v_4v_6:2,3,6;\ \ v_4v_7:4,6,3 ;$\\
$v_5v_6:3,7,4;\ \ v_5v_7:5,3,6;$\\
$v_6v_7:6,2,5.$
\\
Let $V(K_9)=\{v_1,v_2,\ldots,v_9\}$ be the vertex set of $K_9$.
The following colorings, $(c_1,c_2,c_3)$, is a $3$--simultaneous edge coloring of $K_9$,  where $c_\mu$, is a proper edge coloring of $K_9$ with color set $\{1,2,\ldots,9\}$, and  $v_iv_j: l_1,l_2$ means $c_{\mu}(v_iv_j)=l_\mu$,  $\mu=1,2,3$.\\
$v_1v_2:2,3,4;\ \ v_1v_3:1,4,3;\ \ v_1v_4:4,1,2;\ \ v_1v_5:3,2,6;\ \ v_1v_6:6,7,8;\ \ v_1v_7:7,8,5;\ \ v_1v_8:8,5,1;\ \ v_1v_9:5,6,7;$\\
$v_2v_3:9,5,6;\ \ v_2v_4:5,9,7;\ \ v_2v_5:6,7,8;\ \ v_2v_6:3,8,9;\ \ v_2v_7:8,4,2;\ \ v_2v_8:4,6,5;\ \ v_2v_9:7,2,3;$\\
$v_3v_4:6,7,9;\ \ v_3v_5:7,8,5;\ \ v_3v_6:8,6,1;\ \ v_3v_7:4,1,7;\ \ v_3v_8:5,3,8;\ \ v_3v_9:3,9,4;$\\
$v_4v_5:8,6,1;\ \ v_4v_6:7,2,4;\ \ v_4v_7:1,5,8;\ \ v_4v_8:9,8,6;\ \ v_4v_9:2,4,5;$\\
$v_5v_6:9,1,7;\ \ v_5v_7:5,3,9;\ \ v_5v_8:2,9,3;\ \ v_5v_9:1,5,2;$\\
$v_6v_7:2,9,3;\ \ v_6v_8:1,4,2;\ \ v_6v_9:4,3,6;$\\
$v_7v_8:3,2,4;\ \ v_7v_9:9,7,1;$\\
$v_8v_9:6,1,9.$
}\end{proof}

\begin{thm}
Every complete graph $K_n$, except for $n=2,3,5$  admits a $\mu$--simultaneous edge coloring, for $\mu=2,3$.
\end{thm}
\begin{proof}{It is easy to check that $K_2$ and $K_3$ are not $2$--simultaneous edge colorable. In Proposition~\ref{prop:K_5}, we will show that $K_5$, has no $2$--simultaneous edge coloring.
By Theorem~\ref{r-regular}, $K_{2l},$ $l\ge 2$ admits a  $\mu$--simultaneous edge coloring, $\mu=2,3$.
Thus, by Proposition~\ref{prop:K_7} and Theorem~\ref{join}, $K_{11}=K_{7} \vee K_4$ and $K_{13}=K_{7} \vee K_6$ are  $\mu$--simultaneous edge colorable, $\mu=2,3$.  For every $n\ge 14$, we have $K_n=K_{n-4} \vee K_4$; hence  by induction on $n$, and Theorem~\ref{join}, $K_n$ admits a $\mu$--simultaneous edge coloring, $\mu=2,3$.
}\end{proof}

The {\sf Cartesian product} of two graphs $G$ and $H$, denoted by $G\square H$,
is the graph with vertex set $V(G)\times V(H)$ and two vertices $(u,v)$
and $(u',v')$ are  adjacent if and only if either $u=u'$  and  $vv'\in E(H)$
  or  $uu'\in E(G)$  and $v=v'$.


In the following theorems we present some sufficient conditions for $\mu$--simultaneous edge colorable of $G\square H$ in general.
\begin{thm}\label{thm:SEcartesian}
Let $G$ and $H$ be $r$-regular and $s$-regular graphs, respectively.
If $H$ is $1$-factorable, then $G\square H$ is $\mu$--simultaneous edge colorable for
every $\mu\le r+s$.
\end{thm}
\begin{proof}{
Suppose that $G$ and $H$ are $r$-regular and $s$-regular graphs, respectively.
Therefore, $G\square H$ is an $(r+s)$-regular graph.
Since $H$ is $1$-factorable,  
we have $\chi'(H)=\Delta(H)$ and by a theorem in~\cite{Mahmoodian},    
 $\chi'(G\square H)=\Delta(G\square H)=r+s$. Thus
by Theorem~\ref{r-regular}, $G\square H$ is $\mu$--simultaneous edge colorable for
every $\mu\le r+s$.
}\end{proof}
\begin{cor}
\item{\rm (i)}
For every positive integers $n\ge 2$ and $m\ge 3$,
 $C_{2n}\square C_m$ is $\mu$--simultaneous edge colorable for
every $\mu\le 4$.
\item{\rm (ii)}
Let $G$  be $r$-regular. Then,  $G\square K_{2n}$, $n\ge 1$, is $\mu$--simultaneous edge colorable for
every $\mu\le r+2n-1$.
\end{cor}
\begin{thm}
Let $G$ and $H$ be two $\mu$--simultaneous edge colorable graphs. The cartesian product $G\square H$ is also $\mu$--simultaneous edge colorable.
In particular, $\chi'_{\mu-SE}(G\square H)\le \chi'_{\mu-SE}(G)+\chi'_{\mu-SE}(H)$.
\end{thm}
\begin{proof}{
Suppose that $G$ and $H$ be two $\mu$--simultaneous edge colorable graphs.
It is sufficient to consider for each copy of $G$ in $G\square H$ a $\mu$--simultaneous edge coloring
with color set $\{1,\ldots,\chi'_{\mu-SE}(G)\}$ and for each copy of $H$ in $G\square H$ a $\mu$--simultaneous edge coloring
with color set $\{\chi'_{\mu-SE}(G)+1,\chi'_{\mu-SE}(G)+2,\ldots,\chi'_{\mu-SE}(G)+\chi'_{\mu-SE}(H)\}$. Obviously, these colorings form a
 $\mu$--simultaneous edge coloring of $G\square H$.
}\end{proof}

The {\sf lexicographic product} of two simple graphs $G$ and $H$ is the simple graph $G[H]$ whose
vertex set is $V(G)\times V(H)$, and two vertices $(u,v)$ and $(u',v')$ are adjacent if and only if $uu'\in E(G)$,
or $u=u'$ and $vv'\in E(H)$.
\begin{thm}\label{thm:G[H]}
 If  $H$ is $\mu$--simultaneous edge colorable, then for every simple graph $G$, $G[H]$ is also $\mu$--simultaneous edge colorable.
\end{thm}
\begin{proof}{
Let $G$ and $H$ be two simple graphs, $V(G)=\{u_1,\ldots,u_m\}$, and
$V(H)=\{v_1,\ldots,v_n\}$. The graph $G[H]$ consists of
copies $H^1,\ldots,H^m$ of $H$, in which the edge between  $H^i$ and $H^j$ are isomorph  to
 a copy of $K_{n,n}$, whenever $u_iu_j\in E(G)$. Let $J_{ij}$ denote the copy of $K_{n,n}$ corresponds to the edges between
$H^i$ and $H^j$ and $c_G:E(G)\rightarrow \{1,\ldots,\chi'(G)\}$ be a proper edge coloring of $G$.
Now define  $\mu$ edge colorings of $G[H]$. For every $H^i,\ 1\le i\le m$, define
a $\mu$--simultaneous edge coloring the same as $\mu$--simultaneous edge coloring of $H$ by color set $\{1,\ldots,\chi'_{\mu-SE}(H)\}$.
Since by Corollary~\ref{K_r,s}, every $K_{n,n}$ has a $\mu$--simultaneous edge coloring, for every $J_{ij}$
define a $\mu$--simultaneous edge coloring by color set
$\{\chi'_{\mu-SE}(H)+(c_G(u_iu_j)-1)n,\chi'_{\mu-SE}(H)+(c_G(u_iu_j)-1)n+1,\ldots,\chi'_{\mu-SE}(H)+(c_G(u_iu_j)-1)n+(n-1)\}$.
It is easy to check these colorings form a $\mu$--simultaneous edge coloring of $G[H]$.
}\end{proof}
\section{\large  $2$--Simultaneous edge coloring }
In this section we concern on the $2$--Simultaneous edge coloring. First, we study the properties of the extermal counterexample to the strong SE conjecture. Then, we consider the $2$--Simultaneous edge coloring for graphs in general.\\
If the strong  SE conjecture is false, then it must have a minimal counterexample.
We consider the family of counterexamples to the strong SE conjecture with maximum number of vertices among ones with minimum number of edges.
\begin{thm}\label{thm:co.ex.2--SE}
Let $G$ be a bridgeless bipartite graph that is not $2$--simultaneous edge colorable with maximum number of vertices among ones with minimum number of edges, then
\item{\rm (i)} $G$ is $2$-connected;
%
%
\item {\rm (ii)} $\delta(G)=2$ and $\Delta(G)=3$; %
\item {\rm (iii)} $G$ has no nontrivial edge cut of size  $2$;
\item {\rm (iv)} for each $v\in V(G)$, which $\deg(v)=2$, $G - v$ is bridgeless;
\item {\rm (v)} for each $v\in V(G)$, if  $N(v)=\{u,w\}$, then $N(u)\cap N(w)=\{v\}$.
%
\end{thm}
\begin{proof}{
Let $V(G)=X\cup Y$. By Theorem~\ref{SEOCDC}, $G$ is a bridgeless bipartite
graph with no $\rm OCDC$ while every bridgeless bipartite graph $G'$ with $|E(G')|< |E(G)|$
or $|E(G')|= |E(G)|$ and $|V(G')|>|V(G)|$ has an $\rm OCDC$.
\item{\rm (i)} Let $v\in V(G)$ be a cut vertex of $G$. By the minimality of $G$, every block $B$
of $G$ has an $\rm OCDC,\ \mathcal{C}_B$. Therefore,
$$\mathcal{C}=\bigcup_{B\mbox{ is a block of } G}\mathcal{C}_B$$
is an $\rm OCDC$ of $G$, which is a contradiction.
\item{\rm (ii)} Let $v\in V(G)$ be a vertex of degree greater than $3$.
By H. Fleischner's vertex-splitting lemma~\cite{Flieschner},
there exist two edges $e_1=uv$ and $e_2=wv\in E(G)$ such that $G\cup\{uw\}\setminus \{e_1,e_2\}$
 is bridgeless. Let $G'$ be the new graph obtained by subdividing the edge $uw$ in vertex $v'$.
  Thus, $G'$ is bridgeless bipartite graph
such that $|V(G')|=|V(G)|+1$ and $|E(G')|=|E(G)|$. Therefore, $G'$ has an $\rm OCDC,\ \mathcal{C}'$. Let $C_1'$ and $C_2'$ be two directed circuits in $\mathcal{C}'$ that
 include the directed paths $uv'w$ and $wv'u$, respectively.
Define
$C_1=C_1'\cup\{uv,vw\}\setminus \{uv',v'w\}$ and $C_2=C_2'\cup\{wv,vu\}\setminus \{wv',v'u\}$. Then,
$$\mathcal{C}=\mathcal{C}'\cup \{C_1,C_2\}\setminus \{C_1',C_2'\},$$
is an $\rm OCDC$ of $G$, which is a contradiction.

If $\delta(G)\ne 2$, since $G$ is a bridgeless graph and $\Delta(G)\le 3$, $G$ is  $3$-regular. 
Therefore, $G$ is $1$-factorable. Thus by Theorem~\ref{r-regular}, $G$ is $2$--simultaneous colorable, which is a contradiction.
Hence, $\delta(G)=2$ and  by the same reason $\Delta(G)=3$.
\item{\rm (iii)} Let $F=\{e_1=ab,e_2=cd\}$ be a disjoint vertex edge cut of $G$
and $G_1$ and $G_2$ be two nontrivial components of $G\setminus F$ such that $a,c\in V(G_1)$.
Note that the case $a=c$ or $b=d$  does not occur because if so then we get a bridge in $G$.
We consider two following cases.
\\
{$\bullet$} \underline{ $a,d\in X$ and  $b,c\in Y$.}
Let $G_1'=G_1\cup\{ac\}$ and $G_2'=G_2\cup\{bd\}$.
By the edge minimality of $G$, $G_1'$ and $G_2'$ have $\rm OCDC$s, $\mathcal{C}_1$ and $\mathcal{C}_2$,
respectively. Let ${C}_1^1$ and ${C}_1^2$ be two directed circuits in $\mathcal{C}_1$ that
 include the directed edge $ac$ and $ca$, respectively. Assume that
 ${C}_2^1$ and ${C}_2^2$ be two directed circuits in $\mathcal{C}_2$ that
 include the directed edges $db$ and $bd$, respectively.
 Define
$C_1=C_1^1\cup C_2^1\cup\{ab,dc\}\setminus \{ac,db\}$ and
$C_2=C_1^2\cup C_2^2\cup\{ba,cd\}\setminus \{ca,bd\}$, where $uv$ means a directed edge from $u$ to $v$. Thus,
$$\mathcal{C}=\mathcal{C}_1\cup\mathcal{C}_2\cup \{C_1,C_2\}\setminus \{C_1^1,C_1^2,C_2^1,C_2^2\},$$
is an $\rm OCDC$ of $G$, which is a contradiction.
\\
{$\bullet$} \underline{ $a,c\in X$ and  $b,d\in Y$.}
Let $G_1'$ be the graph obtained from $G_1$ by joining a new vertex $v_1$ to $a$ and $c$,
and  $G_2'$ be the graph obtained from $G_2$ by joining a new vertex $v_2$ to $b$ and $d$.
By the edge minimality of $G$, bipartite graphs $G_1'$ and $G_2'$ have $\rm OCDC$s, $\mathcal{C}_1$ and $\mathcal{C}_2$,
respectively. Let ${C}_1^1$ and ${C}_1^2$ be two directed circuits in $\mathcal{C}_1$ that
 include the directed paths $av_1c$ and $cv_1a$, respectively. Assume that
 ${C}_2^1$ and ${C}_2^2$ be two directed circuits in $\mathcal{C}_2$ that
 include the directed paths $dv_2b$ and $bv_2d$, respectively.
 Define
$C_1=C_1^1\cup C_2^1\cup\{ab,dc\}\setminus \{av_1,v_1c,dv_2,v_2b\}$ and
$C_2=C_1^2\cup C_2^2\cup\{ba,cd\}\setminus \{v_1a,cv_1,v_2d,bv_2\}$, where $uv$ means a directed edge from $u$ to $v$. Thus,
$$\mathcal{C}=\mathcal{C}_1\cup\mathcal{C}_2\cup \{C_1,C_2\}\setminus \{C_1^1,C_1^2,C_2^1,C_2^2\},$$
is an $\rm OCDC$ of $G$, which is a contradiction.

\item{\rm (iv)} If $\deg(v)=2$, then every bridge in $G-v$    with one of the edges incident on $v$ forms an nontrivial edge cut of size $2$, which is a contradiction.
\item{\rm (v)} Suppose that  $N(v)=\{u,w\}$ and $v'\in (N(u)\cap N(w))\setminus \{v\}$.
By {\rm (iv)} and the minimality of $G$, $G - v $ has an $\rm OCDC$, $\mathcal{C}'$. Since $\deg_G(v')\le 3$, without loss of generality,
there exists a directed circuit $C\in \mathcal{C}'$ that include the directed edges $uv'$ and $v'w$.
Let
$C_1=C\cup\{uv,vw\}\setminus \{v'\}$ and $C_2=vuv'wv$. Then,
$$\mathcal{C}=\mathcal{C}'\cup \{C_1,C_2\}\setminus \{C\},$$
is an $\rm OCDC$ of $G$, which is a contradiction.
}\end{proof}
In the rest of this section, we consider the $2$--simultaneous edge coloring for graphs in general.
For example the following two colorings is a $2$--simultaneous edge coloring for wheel $W_n$, $n\ge 3$.
Assume that $V(W_n)=\{u,v_1,\ldots,v_n\}$ and
$E(W_n)=\{uv_i,v_iv_{i+1} (\mbox{mod } n) : 1\le i\le n\}$.
Define two edge coloring $f_j:E(W_n) \rightarrow [n]$, $j=1,2$,
$f_1(uv_i)=i$, $f_1(v_iv_{i+1})=i+2$, and $f_2(uv_i)=i+2$, $f_2(v_iv_{i+1})=i+1$, where the colors and subscripts are reduced modulo $n$.
It is easy to check that $(f_1,f_2)$ forms a $2$--simultaneous edge coloring of $W_n$.
\begin{thm}\label{2k--subdivide}
Let $G$ be a $2$--simultaneous edge colorable graph. If $G'$ is a graph obtained from $G$ by replacing
an edge $xy\in E(G)$ with simple path $xv_1v_2\ldots v_{2k}y$ such that $v_i\notin V(G),\ 1\le i\le2k$,
then $G'$ is also $2$--simultaneous edge colorable.
\end{thm}
\begin{proof}{
Let $(f_1,f_2)$ be a $2$--simultaneous edge coloring of $G$. Without loss of generality, suppose that
$f_j(xy)=j,\ j=1,2$. Define two proper edge colorings $f'_1$ and $f'_2$ of $G'$ as follows.
$f'_j(xv_1)=f'_j(v_{2i}v_{2i+1})=f'_j(v_{2k}y)=j$, $f'_j(v_{2i-1}v_{2i})=j+1\ (\rm mod\ 2)$, and
$f'_j(e)=f_j(e)$ for $e\in E(G)\setminus \{xy\},\ 1\le i\le k-1,$ and $j=1,2$.
Therefore, $(f'_1,f'_2)$ is  a $2$--simultaneous edge coloring of $G'$.
}\end{proof}
\begin{thm}
Let $G$ be a bridgeless graph with girth at least $2k-1$, $k\ge 2$. If $G$ is $2$--simultaneous edge colorable, then $|E(G)|\ge k\chi'(G)$.
\end{thm}
\begin{proof}{
Let $(f_1,f_2)$ be a $2$--simultaneous edge coloring of $G$ and $f_i^j=\{e\in E(G) : f_i(e)=j\}$, $i=1,2$.
Since $\chi'(G)\le \chi'_{2-SE}(G)$, if $|E(G)|< k\chi'(G)$,  then for some $j,\ 1\le j\le \chi'_{2-SE}(G)$, $|f_i^j|\le k-1$ for $i=1,2$.
Therefore, the induced subgraph by $f_1^j\cup f_2^j$ is a union of even circuits of length  at most  $2k-2$,   which is a contradiction.
}\end{proof}
%
%

In the following, we provide a relation between $2$--simultaneous edge coloring of a graph with a $\rm CDC$ with certain properties. Then, using this relation, we show some $2$--simultaneous edge colorable graph and also some graph which has no $2$--simultaneous edge coloring.
\begin{thm}\label{SE-CDC}
A bridgeless graph $G$ is a $2$--simultaneous edge colorable if and only if $G$ has a $\rm CDC,\ \mathcal{C},$ that satisfies in the following properties.
\item{\rm (i)} Every circuit of $\mathcal{C}$ is an even circuit.
\item{\rm (ii)} $\mathcal{C}$ has a partition to at least $\chi'(G)$ classes,
such that every class is $2$-regular.
\item{\rm (iii)} Every circuit in $\mathcal{C}$ has a proper $2$-edge coloring, such that each edge $e\in E(G)$ in different circuits
admits two different colors.
\end{thm}
\begin{proof}{
Suppose that $(f_1,f_2)$  is a $2$--simultaneous edge coloring of graph $G$.
Let $f_i^j=\{e\in E(G) : f_i(e)=j\}$ for $1\le j\le \chi'_{2-SE}(G)$ and $i=1,2$.
The induced subgraph by $C_j=f_1^j\cup f_2^j$ is a disjoint union of even circuits, for $1\le j\le \chi'_{2-SE}(G)$.
Therefore, $\mathcal{C}=\{C_j : 1\le j\le\chi'_{2-SE}(G)\}$ is a $\rm CDC$ of $G$ with properties {\rm (i)} and {\rm (ii)}.
Now for every edge $e$ in a circuit of $C_j$, let $c_j(e)=i$, where $f_i(e)=j,$ $i=1,2$, one can see that $c_j$ satisfies the property {\rm (iii)}.

 Conversely, let $\mathcal{C}$ be a $\rm CDC$, where $C_1,C_2,\ldots,C_t$ is a partition of $\mathcal{C}$ such that $C_i$, $1\le i\le t$, is $2$-regular
and $c_i$ is a proper edge coloring of $C_i$ satisfies in condition {\rm (iii)}.
Now we define two edge colorings $(f_1,f_2)$ as follows. For every edge $e$, if
$e\in C_j$ and $c_j(e)=i$, then set $f_i(e)=j$. By the assumption, it is clear that $f_i$, $i=1,2$, is a proper edge colorings and
$f_1(e)\ne f_2(e)$ for every $e\in E(G)$. It is enough to show that the set of colors appear on the edges incident to each vertex are the same.
Let $v$ be an arbitrary vertex of $G$ and $u\in V(G)$ be an arbitrary neighbor of $v$.
Without loss of generality, suppose that $f_1(uv)=j$, $1\le j\le t$,
$uv\in C_j$ and $c_j(uv)=1$. Since  $C_j$ is $2$-regular and $c_j$ is a proper $2$-edge coloring,
there exists an edge $vw\in C_j$ that $c_j(vw)=2$. Therefore, $f_2(vw)=j$. Thus, $(f_1,f_2)$ is
a $2$--simultaneous edge coloring of $G$.
}\end{proof}
If $G$ has an even circuit decomposition, then two copies of this decomposition  satisfies in three conditions of
Theorem~\ref{SE-CDC}. Hence, $G$ is $2$--simultaneous edge colorable. In other words, by Theorem~\ref{SE-CDC}, an even graph, $G$ is
$2$--simultaneous edge colorable if and only if $G$ has an even circuit decomposition.
\begin{thm}\label{Hamil.OCDC}
Let $C$ be an even Hamiltonian circuit of $G$ and $G\setminus E(C)$ be a bipartite graph. If
$G\setminus E(C)$ has an $\rm OCDC$, then $G$ has a $2$--simultaneous edge coloring.
\end{thm}
\begin{proof}{
By Theorem~\ref{SEOCDC}, $G\setminus E(C)$ is $2$--simultaneous edge colorable. Therefore  by Theorem~\ref{SE-CDC}, it has a $\rm CDC$, $\mathcal{C'}$, of even circuits
that has a partition to  even $2$-regular subgraphs and a proper $2$-edge coloring such that each edge of $G\setminus E(C)$ admits two different colors. Now let $\mathcal{C}=\mathcal{C'}\cup \{C,C\}$. It is easily seen that, $\mathcal{C}$ satisfies in three conditions of Theorem~\ref{SE-CDC}.
Thus,   $G$ is $2$--simultaneous edge colorable.
}\end{proof}
By Theorem~\ref{thm:4-NZF}~(i) and~(iii), we have the following corollary.
\begin{cor}
Let $C$ be an even Hamiltonian circuit of $G$ and $G\setminus E(C)$ be a bipartite graph.
If every edge of  $G\setminus E(C)$ is contained in a circuit of length  $4$ in $G\setminus E(C)$,
then $G$ is a $2$--simultaneous edge colorable graph.
\end{cor}
\begin{prop}\label{prop:K_5}
The complete graph $K_5$ has no $2$--simultaneous edge coloring.
\end{prop}
\begin{proof}{
Let $(f_1,f_2)$ be a $2$--simultaneous edge coloring of $K_5$ and $f_i^j=\{e\in E(G) : f_i(e)=j\}$, $i=1,2$ and $1\le j\le 5$. Since $\chi'(K_5)=5$ and $|E(K_5)|=10$, the induced subgraph by
 $f_1^j\cup f_2^j$ is a circuit of length $4$, for $1\le j\le 5$. By the isomorphic, there is exactly one {\rm CDC}
 of $K_5$ with even circuits, see Figure~\ref{K-5CDC}. It is easy to check that the condition {\rm (}iii{\rm)}
 of Theorem~\ref{SE-CDC}
 does not hold for this {\rm CDC}, which is a contradiction.
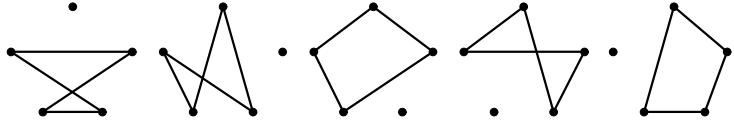
\begin{figure}[ht]
\centering \unitlength=1mm
\begin{picture}(109,26)
\drawthickdot{27.04}{15.00}
\drawthickdot{6.80}{15.00}
\drawthickdot{22.95}{15.00}
\drawthickdot{42.93}{15.00}
\drawthickdot{47.04}{15.00}
\drawthickdot{62.91}{15.00}
\drawthickdot{67.01}{15.00}
\drawthickdot{83.04}{15.00}
\drawthickdot{86.87}{15.00}
\drawthickdot{102.0}{15.00}
\drawthickdot{15.01}{21.00}
\drawthickdot{35.00}{21.00}
\drawthickdot{54.97}{21.00}
\drawthickdot{74.97}{21.00}
\drawthickdot{94.94}{21.00}
\drawthickdot{31.02}{7.00}
\drawthickdot{38.97}{7.00}
\drawthickdot{51.00}{7.00}
\drawthickdot{58.81}{7.00}
\drawthickdot{11.03}{7.00}
\drawthickdot{18.97}{7.00}
\drawthickdot{71.00}{7.00}
\drawthickdot{78.94}{7.00}
\drawthickdot{90.98}{7.00}
\drawthickdot{99.05}{7.00}
\thicklines
\drawpath{11.03}{7.00}{18.97}{7.00}
\drawpath{22.95}{15.00}{11.03}{7.00}
\drawpath{6.80}{15.00}{22.95}{15.00}
\drawpath{18.97}{7.00}{6.80}{15.00}
\drawpath{27.04}{15.00}{31.02}{7.00}
\drawpath{31.02}{7.00}{35.00}{21.00}
\drawpath{35.00}{21.00}{38.97}{7.00}
\drawpath{38.97}{7.00}{27.04}{15.00}
\drawpath{54.97}{21.00}{47.04}{15.00}
\drawpath{47.04}{15.00}{51.00}{7.00}
\drawpath{51.00}{7.00}{62.91}{15.00}
\drawpath{62.91}{15.00}{54.97}{21.00}
\drawpath{74.97}{21.00}{67.01}{15.00}
\drawpath{67.01}{15.00}{83.04}{15.00}
\drawpath{83.04}{15.00}{78.94}{7.00}
\drawpath{78.94}{7.00}{74.97}{21.00}
\drawpath{94.94}{21.00}{90.98}{7.00}
\drawpath{90.98}{7.00}{99.05}{7.00}
\drawpath{99.05}{7.00}{102.0}{15.00}
\drawpath{102.0}{15.00}{94.94}{21.00}
\end{picture}
\caption{\label{K-5CDC} \scriptsize{A {\rm CDC} of $K_5$ with even circuits.}}
\end{figure}
}\end{proof}
Since the Petersen graph has no $4$-$\rm NZF$~\cite{Zhang'sBook}, we conclude the following Theorem.
\begin{prop}
The Petersen graph is not $2$--simultaneous edge colorable.
\end{prop}
\begin{proof}{
By the contrary,
let the Petersen graph, $P$, is $2$--simultaneous edge colorable.
Let $(f_1,f_2)$ be a $2$--simultaneous edge coloring of $P$ and $f_i^j=\{e\in E(G) : f_i(e)=j\}$, $i=1,2$.
Thus, $P$ has a $\rm CDC$ of even cicuits.
Since $C_6$ and $C_8$  are
only even circuits in $P$ and $|E(P)|=15$, without loss of generality, the induced subgraph by $f_1^j\cup f_2^j$ is a circuit of length $8$ for $j=1,2,3$
and the induced subgraph by $f_1^4\cup f_2^4$ is a circuit of length $6$ or the induced subgraph by $f_1^j\cup f_2^j$ is a circuit of length $6$ for $j=1,\ldots,5$. It is easy to check that the second case is not possible.
 In the first case,
$\mathcal{C}=\{C_j=f_1^j\cup f_2^j\ :\ 1\le j\le4\}$ is a $4$-$\rm CDC$ of $P$.
Therefore by Theorem~\ref{thm:4-NZF}~(ii), $P$ admits a $4$-$\rm NZF$, while it has not~\cite{Zhang'sBook}.
}\end{proof}

\section*{Acknowledgments}
The authors thank Professor C.-Q. Zhang
for his  helpful suggestions and also, they   appreciate
 the help of Amir Hooshang Hosseinpoor
 for his computer programming.



\begin{thebibliography}{99}
%
\bibitem{MR1904722}
{\it P. Adams, E.J. Billington, D.E. Bryant and E.S. Mahmoodian}, {\it On the possible volumes of {$\mu$}--way {L}atin trades},
Aequationes Math., {\bf63(3)}, (2002) 303-320.
%
\bibitem{MR1874724}
{\it P. Adams, E.J. Billington, D.E. Bryant and E.S. Mahmoodian}, {\it The
three--way intersection problem for Latin squares}, Discrete Math., {\bf243} (2002) 1-19.
%
%
\bibitem{MR2041871}
{\it E.J. Billington.} {\it Combinatorial trades: a survey of recent results},
In Designs, 2002, volume 563 of  Math. Appl.,
   Kluwer Acad. Publ., Boston, MA, (2003) 47-67.

\bibitem{Bondy}
{\it J.A. Bondy and U.S.R. Murty},{ Graph Theory,} {\it Springer,} (2008).
 %
%
\bibitem{MR2453264}
{\it N.J. Cavenagh}, {\it The theory and application of {L}atin bitrades: a survey},
 Math. Slovaca, {\bf58(6)} (2008) 691-718.
 %
 \bibitem{Edmonds}
{\it J. Edmonds}, {\it Existence of $k$-edge connected ordinary graphs with prescribed degrees},
 Journal of Research of the National Bureau of Standards, {\bf68B(2)} (1964) 73-74.
%
\bibitem{Flieschner}
{\it H. Fleischner}, {\it Eine gemeinsame Basis f$\ddot{u}$r die Theorie der Eulerschen
Graphen und den Satz von Petersen}, Monatsh. Math., {\bf81} (1976)  126-130.
%
\bibitem{Hajiaghaee}
{\it M.T. Hajiaghaee, E.S. Mahmoodian, V.S. Mirrokni, A. Saberi, and R. Tusserkani}, {\it On the simultaneous edge-coloring conjecture}, Discrete Math., {\bf216} (2000) 267-272.
\bibitem{Survey} {\it  F. Jaeger}, {\it A survey of the cycle double cover cojecture},
in: B. Alspach and C. Godsil, eds., Cycles in Graphs; Ann., Discrete Math., {\bf 27} (1985) 1-12.
%
\bibitem{MR2048415}
{\it A.~D. Keedwell},
{\it Critical sets in {L}atin squares and related matters: an update.}
Util. Math., {\bf65} (2004) 97-131.
%
\bibitem{4NZF--SE}
{\it R. Luo, W. Zang, and C.-Q. Zhang}, {\it Nowhere-zero $4$-flows, simultaneous edge colorings, and critical partial {L}atin squares},
 Combinatorica, {\bf24(4)} (2004) 641-657.
\bibitem{Mahmoodian}
{\it E.S. Mahmoodian}, {\it On edge-colorability of Cartesian products of graphs}, Canad. Math. Bull., {\bf24} (1981) 107-108.
%
\bibitem{Mahdian}
{\it M. Mahdian, E.S. Mahmoodian, A. Saberi, M.R. Salavatipour, and R. Tusserkani}, {\it On a conjecture of Keedwell and
the cycle double cover conjecture},
  Discrete Math., {\bf216(1-3)}(2000) 287-292.
%
%
\bibitem{Sey}
 {\it P.D. Seymour}, Sums of circuits, In: J.A. Bondy, U.S.R. Murty (Eds.), Graph Theory and Related Topics,
{\it Academic Press, New York}, (1979) 341-355.
%
%
%
%
\bibitem{Zhang'sBook}
{\it C.-Q. Zhang},{ Integer flows and cycle covers of graphs,} {\it Marcel Dekker,  Inc., New York} (1997).

\end{thebibliography}
\end{document}